\documentclass{amsart}
\usepackage[utf8]{inputenc}
\usepackage[english]{babel}
\usepackage{amsthm}
\usepackage{amsfonts}         
\usepackage{amsmath}
\usepackage{amssymb}
\usepackage{enumerate}
\usepackage[normalem]{ulem}
\usepackage{comment}

\newcommand{\R}{\mathbb{R}}
\newcommand{\C}{\mathbb{C}}

\newcommand{\N}{\mathbb{N}}

\renewcommand{\S}{\mathcal{S}}

\newcommand{\vspan}{\operatorname{span}}

\DeclareMathOperator{\tr}{tr}

\makeatletter
\def\th@plain{%
  \thm@notefont{}
  \itshape 
}
\def\th@definition{%
  \thm@notefont{}
  \normalfont 
}
\makeatother

\theoremstyle{plain}
\newtheorem{lause}[subsubsection]{Theorem}
\newtheorem{lem}[subsubsection]{Lemma}
\newtheorem{kor}[subsubsection]{Corollary}

\theoremstyle{definition}
\newtheorem{konj}[subsubsection]{Conjecture}
\newtheorem{kys}[subsubsection]{Question}

\theoremstyle{remark}

\title{Planes in Schatten-$3$}
\author{Otte Heinävaara}
\address{Department of Mathematics, Princeton University, Princeton, NJ 08544}
\email{oeh@math.princeton.edu}

\begin{document}

\begin{abstract}
    We prove that any two-dimensional real subspace of Schatten-3 can be linearly isometrically embedded into $L_{3}$. This resolves the case $p = 3$ of Hanner's inequality for Schatten classes, conjectured by Ball, Carlen and Lieb. We conjecture that similar isometric embedding is possible for any $p \geq 1$.
\end{abstract}

\maketitle

\section{Introduction}

For $p \geq 1$, let $\S_{p}$ denote the Schatten-$p$ trace-class; the space of compact operators on complex separable Hilbert space $H$ for which the sequence of singular values  $(\sigma_{i}(A))_{i = 1}^{\infty}$ belongs to $\ell_{p}$, with the norm
\begin{align}\label{schattennorm}
    \|A\|_{\S_{p}} = \left(\sum_{i = 1}^{\infty} \sigma_{i}(A)^{p}\right)^{1/p} = \left(\tr (A^{*}A)^{p/2}\right)^{1/p}.
\end{align}

Our main result is the following.

\begin{lause}\label{mainemb}
    For any $A, B \in \S_{3}$ the real space $\vspan_{\R}\{A, B\}$ (with $\S_{3}$ norm) can be linearly isometrically embedded into $L_{3}$.
\end{lause}

In other words, for any $A, B \in \S_{3}$ (or $n \times n$ matrices $A$ and $B$) there exits $f, g \in L_{3}(0, 1)$ such that for any $\alpha, \beta \in \R$ one has
\begin{align*}
    \|\alpha A + \beta B\|_{\S_{3}} = \|\alpha f + \beta g\|_{L_{3}}.
\end{align*}

As a consequence, to prove any inequality that only depends on $\S_{3}$-norms of real linear combinations of two matrices, it is enough to consider diagonal matrices. This implies, for instance, that $\S_{3}$ has the same modulus of uniform convexity and smoothness as $\ell_{3}$, fact first proven by Ball, Carlen and Lieb in 1994 \cite{ball1994sharp}.

Theorem \ref{mainemb} also implies Hanner's inequality for $\S_{3}$. Namely, for any pair of $n \times n$ matrices $A$ and $B$, the following inequality holds for $p = 3$:
\begin{align}\label{hanner1}
    \|A + B\|_{\S_p}^{p} + \|A - B\|_{\S_p}^{p} \leq (\|A\|_{\S_p} + \|B\|_{\S_p})^{p} + |\|A\|_{\S_p} - \|B\|_{\S_p}|^{p}.
\end{align}

This inequality was first considered by Hanner in 1955 \cite{hanner1956uniform} for $L_{p}$. Hanner proved that for $p \geq 2$ and $f, g \in L_{p}(0, 1)$ one has
\begin{align}
    \|f + g\|_{L_p}^{p} + \|f - g\|_{L_p}^{p} \leq (\|f\|_{L_p} + \|g\|_{L_p})^{p} + |\|f\|_{L_p} - \|g\|_{L_p}|^{p}.
\end{align}
Hanner also proved that the reverse inequality holds for $1 \leq p \leq 2$. These inequalities were then used to estimate the modulus of uniform convexity of $L_{p}$.

The non-commutative version (\ref{hanner1}) was first discussed by Ball, Carlen and Lieb in \cite{ball1994sharp}; the inequality (\ref{hanner1}) was proved for $p \geq 4$ and the reverse inequality, i.e. for any pair of $n \times n$ matrices $A$ and $B$
\begin{align}\label{hanner2}
    \|A + B\|_{\S_p}^{p} + \|A - B\|_{\S_p}^{p} \geq (\|A\|_{\S_p} + \|B\|_{\S_p})^{p} + |\|A\|_{\S_p} - \|B\|_{\S_p}|^{p}
\end{align}
was then proved by duality (see \cite[Lemma $6$]{ball1994sharp}) for $1 \leq p \leq 4/3$. While the authors managed to find the modulus of uniform convexity of $\S_{p}$ for every $1 < p < \infty$ using other means, it was conjectured that the inequalities (\ref{hanner1}) and (\ref{hanner2}) hold in full ranges $p \geq 2$ and $1 \leq p \leq 2$, respectively. Despite serious effort, these ranges of $p$ have not been improved until now\footnote{See however \cite{mc1967c_rho}, \cite{ball1994sharp} and \cite{chayes2020matrix} for partial results with special assumptions on $A$ and $B$.}. In addition to the case $p = 3$, the aforementioned duality argument then also implies (\ref{hanner2}) for $p = 3/2$. Understanding the non-commutative Hanner's inequalities (\ref{hanner1}) and (\ref{hanner2}) was the original motivation for the main result Theorem \ref{mainemb}.

We make the following obvious conjecture.

\begin{konj}\label{mainconj}
    For any $p \geq 1$ and $A, B \in \S_{p}$ the real space $\vspan_{\R}\{A, B\}$ can be linearly isometrically embedded into $L_{p}$.
\end{konj}

If true, this would imply the non-commutative Hanner's inequalities (\ref{hanner1}) and (\ref{hanner2}) for full ranges of $p$. More generally, to prove any $\S_{p}$-inequality depending only on norms of elements in a plane, it would be enough to consider commuting/diagonal operators. See section \ref{discuss} for some discussion and further partial results towards this conjecture. 

We also note that similar isometric embedding is not possible in general for $3$-dimensional subspaces, for any $p \geq 1$.

\begin{lause}\label{negthm}
    For any $1 \leq p < \infty$, $p \neq 2$ there exists a $3$-dimensional subspace of $\S_{p}$ which is not linearly isometric to any subspace of $L_{p}$. In fact one may consider the space of real symmetric $2 \times 2$ matrices.
\end{lause}

The main result Theorem \ref{mainemb} is closely related to the following new result of independent interest.

\begin{lause}\label{der4}
    Let $f : (a, b) \to \R$ be a function with non-negative $4$th derivative (in distributional sense). Then for any pair of Hermitian $n \times n$ matrices $A$ and $B$ the function
    \begin{align*}
        t \mapsto \tr f(A + t B)
    \end{align*}
    also has non-negative $4$th derivative.
\end{lause}

An analogous result was known earlier for second derivatives (convex functions), and first derivatives (increasing functions) with the additional assumption that $B$ is positive semidefinite (see for instance \cite[Propositions $1$ and $2$]{petz1994survey}). We also prove similar result for third derivatives, and conjecture that an analogous statement holds true for derivatives of any order, under the additional assumption that $B$ is positive semidefinite for odd orders.

\begin{lause}\label{der3}
    Let $f : (a, b) \to \R$ be a function with non-negative $3$rd derivative (in distributional sense). Then for any $n \times n$ matrices $A, B$ with $B \geq 0$ the function
    \begin{align*}
        t \mapsto \tr f(A + t B)
    \end{align*}
    also has non-negative $3$rd derivative.
\end{lause}

\begin{konj}\label{der_conjecture}
    Let $k$ be positive integer and $f : (a, b) \to \R$ a function with non-negative $k$'th derivative (in distributional sense). Let $A, B$ be $n \times n$ matrices, and if $k$ is odd, additionally assume that $B$ is positive semidefinite. Then the function
    \begin{align*}
        t \mapsto \tr f(A + t B)
    \end{align*}
    has non-negative $k$'th derivative.
\end{konj}

In section \ref{discuss} we show that if true, conjecture \ref{der_conjecture} would give a new proof (and be a vast generalization of) Stahl's Theorem \cite{stahl2013proof} (previously known as BMV conjecture \cite{bessis1975monotonic}).

\begin{lause}[Stahl]\label{stahl}
    Let $A, B$ be Hermitian $n \times n$ matrices with $B$ positive semidefinite. Then
    \begin{align*}
        t \mapsto \tr(\exp(A - t B))
    \end{align*}
    is a Laplace transform of a positive measure on $[0, \infty)$.
\end{lause}

Section \ref{mainsec} contains the proofs of the main results, together with the necessary preliminaries. Section \ref{discuss} contains discussion and partial results towards the aforementioned conjectures.

\section{Proofs of the main results}\label{mainsec}

\subsection{Preliminaries}\label{prel}

We make the usual conventions $\ell_{p} := L_{p}(\N, \R)$ and $L_{p} := L_{p}((0, 1), \R)$, with counting and Lebesgue measures, respectively. We say that a norm $\|\cdot\|$ on $\R^{n}$ is $L_{p}$-norm, if $(\R^{n}, \|\cdot\|)$ is isometric to a subspace of $L_{p}$. This is equivalent to the existence of a symmetric measure $\mu$ on $S^{n - 1}$ (given any inner product $\langle \cdot, \cdot \rangle$) such that
\begin{align}\label{lp_repr}
    \|v\|^{p} = \int_{S^{n - 1}} |\langle v, u \rangle|^{p} d \mu(u)
\end{align}
for any $v \in \R^{n}$ (see \cite{neyman1984representation}). If $p$ is not an even integer, such measure, whenever it exists, is unique \cite[Theorem 2]{koldobskii1991convolution}. The issue of recovering the measure from the norm is discussed by Koldobsky in \cite{koldobsky1992generalized}; we shall need a simple corollary of the work therein.

\begin{lem}\label{lemKol}
	Let $\|\cdot\|$ be a norm in $\R^{2}$ such that for any $v, w \in \R^{2}$ the $4$th distributional derivative of
	\begin{align*}
		t \mapsto \|t v + w\|^{3}
	\end{align*}
	is a finite positive Borel measure $\mu$, for which $\mu(|x|^3) < \infty$. Then $\|\cdot\|$ is linearly isometrically embeddable into $L_{3}$.
\end{lem}

See section \ref{discuss} for version of Lemma \ref{lemKol} for general $p \geq 1$.

\begin{proof}[Proof of Lemma \ref{lemKol}]
	It follows directly Theorem $5$ of \cite{koldobsky1992generalized} and Remark $1$ thereafter that for any $v, w \in \R^{2}$ there exists two functions $h_{v}, h_{w} : \R \to \R$, with $u_{v}$ polynomial, $h_{v}(0) = 0 = h_{w}(0)$; and $L_{p}$ norm $\|\cdot\|_{3, v, w}$ on $\R^{2}$ such that
	\begin{align*}
		\|x v + y w\|^{3} = \|x v + y w\|_{3, v, w}^{3} + h_{v}(x) + h_{w}(y).
	\end{align*}
	Considering $y = 0$, one sees that $h_{v}$ can only be polynomial if $h_{v} = 0$; similarly $u_{w}(y) = \alpha |y|^{3}$ for some $\alpha \in \R$. But this means that
	\begin{align*}
	    \|\cdot\|^{3} = \int_{S^{1}} |\langle \cdot, u \rangle|^{3} d \mu(u),
	\end{align*}
	where $\mu$ is positive measure, except possibly at one point (depending on $v, w$ and the pairing). Changing $v, w$ and applying the uniqueness result \cite[Theorem 2]{koldobskii1991convolution}, we see that $\mu$ is positive everywhere.
\end{proof}

We shall need only very basic properties Schatten classes. Together with the norm (\ref{schattennorm}) they form Banach spaces. For $1 \leq p < \infty$ finite rank operators are dense in $\S_{p}$ \cite{mc1967c_rho}. While $\ell_{p}$ is isometric a subspace of $\S_{p}$, $\S_{p}$ is not isomorphic to any subspace of $L_{p}$ \cite{pisier1978some}.

We shall need basic properties of divided differences. Define
\begin{align*}
	[x_{0}]_{f} = f(x_{0}),
\end{align*}
and recursively for any positive integer $k$
\begin{align}\label{lol}
	[x_{0}, x_{1}, \ldots, x_{k}]_{f} = \frac{[x_{0}, x_{1}, \ldots, x_{k - 1}]_{f} - [x_{1}, x_{2}, \ldots, x_{k}]_{f}}{x_{0} - x_{k}},
\end{align}
when points $x_{0}, x_{1}, \ldots, x_{k}$ are pairwise distinct. If $f$ is $C^{k}$, the divided differences of order $k$ has continuous extension to all tuples of $(k + 1)$ points. This extension satisfies (\ref{lol}) whenever $x_{0} \neq x_{k}$ and
\begin{align*}
	[x_{0}, x_{0}, \ldots, x_{0}]_{f} = \frac{f^{(k)}(x_{0})}{k!},
\end{align*}
where $x_{0}$ appears $(k + 1)$ times. For this and lot more, see for instance \cite{de2005divided}.

\subsection{Proofs of Theorems \ref{mainemb} and \ref{negthm}}

Our goal is to check the conditions of Lemma \ref{lemKol}.

\begin{lem}\label{lem_decay}
	Let $A$ and $B$ be $n \times n$ Hermitian matrices such that $B$ is invertible. Then
	\begin{align*}
		\left(\tr |A + t B|^{3}\right)^{(4)} = O(1/|t|^{5})
	\end{align*}
	when $|t| \to \infty$.
\end{lem}
\begin{proof}
	Recall that since $A$ and $B$ are Hermitian, by the result of Rellich \cite{rellich1969perturbation} the eigenvalues of $B + \varepsilon A$ are $n$ analytic functions with fixed sign for small enough $\varepsilon$. Consequently for large enough $t$ one has
	\begin{align*}
		\tr|A + t B|^{3} = \sum_{i = 1}^{n} |t|^{3} |\lambda_{i}(B)|^{3} (1 + O(1/t)),
	\end{align*}
	where the multipliers are analytic outside compact set. Differentiating this $4$ times kills the polynomial part and each summand is $O(1/|t|^{5})$ as required.
\end{proof}

\begin{lause}\label{posres}
	For any Hermitian $n \times n$ matrices $A$ and $B$, with $B$ invertible, the function
	\begin{align*}
		T := t \mapsto \tr |A + t B|^{3}
	\end{align*}
	is analytic outside finitely many points where $A + t B$ is singular. Outside these points the function has non-negative $4$th derivative. At the singular points the function behaves like
	\begin{align*}
		C |t - c|^{3} + D |t - c|^{3} (t - c) + f(t),
	\end{align*}
	where $C \geq 0$ and $f$ is $C^{4}$ near $c$. Consequently the $4$th derivative has non-negative multiple of $\delta$ measure at the singular points is therefore altogether non-negative measure.
\end{lause}

\begin{proof}
	The second claim follows straightforwardly from the analyticity of the eigenvalues. Indeed, assume that say $0$ is a singular point. Now any eigenvalue can be written analytically as $\lambda_{i}(t)$. If $\lambda_{i}(0) \neq 0$, $|\lambda_{i}(t)|^{3}$ is smooth near $0$. If on the other hand $\lambda_{i}(0) = 0$, we can write $\lambda_{i}(t) = t^{k} \nu_{i}(t)$ for some analytic $\nu_{i}$ with $\nu_{i}(0) \neq 0$ and $k$ positive integer (recall that $B$ is invertible). Now
	\begin{align*}
		|\lambda_{i}(t)|^{3} = |t|^{3 k} |\nu_{i}(t)|^{3}.
	\end{align*}
	Note however that the second term is analytic near $0$. If $k > 1$ the whole function is $C^{4}$ near $0$. If $k = 1$, we may expand $|\nu_{i}|^{3}$ at $0$ to get required expansion for a single eigenvalue. Repeating this for all the eigenvalues yields the claim.

	Let us then move to the heart of the matter. First note that since $\det(A t + B)$ is non-zero polynomial, it can only have finitely many zeroes, and there can consequently be only finitely many points where $T$ is not analytic.

	Next we need an identity for the smooth part.

	\begin{lem}\label{derivative_identity}
		(cf. \cite[Theorem 2.3.1.]{hiai2010matrix}) Let $A, B$ be Hermitian $n \times n$ matrices and $f$ be analytic near the eigenvalues of $A$. Write $F(t) = \tr(f(A + t B))$. Then
		\begin{align}\label{main_expr}
			\frac{F^{(k)}(0)}{k!} = \frac{1}{k}\sum_{i_{1} = 1}^{n} \sum_{i_{2} = 1}^{n} \cdots \sum_{i_{k} = 1}^{n} [\lambda_{i_{1}}, \lambda_{i_{2}}, \ldots, \lambda_{i_{k}}]_{f'} B_{i_{1}, i_{2}} B_{i_{2}, i_{3}} \cdots B_{i_{k}, i_{1}}.
		\end{align}
		Here $B_{i, j}$ is the matrix of $B$ in the eigenbasis of $A$, and $\lambda_{1}, \lambda_{2}, \ldots, \lambda_{n}$ are the eigenvalues of $A$.
	\end{lem}
	\begin{proof}
		The LHS clearly only depends on the value of $f$ and its first $k$ derivatives at the eigenvalues of $A$. Hence it is sufficient to prove such identity for a special class of functions, say for polynomials. For function $(\cdot)^{m}$ the LHS is simply
		\begin{align*}
			(\tr(A + t B)^{m})^{(k)}(0) =& \sum_{j_{0}, j_{1}, \ldots,  j_{k} \geq 0, \sum j_{i} = m - k} \tr(A^{j_{0}} B A^{j_{1}} B \cdots B A^{j_{k}}) \\
			=& \sum_{j_{1}, \ldots,  j_{k} \geq 0, \sum j_{i} = m - k} (j_{k} + 1) \tr(B A^{j_{1}} B \cdots B A^{j_{k}}) \\
			=& \frac{m}{k} \sum_{j_{1}, \ldots,  j_{k} \geq 0, \sum j_{i} = m - k} \tr(A^{j_{1}} B \cdots B A^{j_{k}} B) \\
			=& \frac{m}{k} \sum_{j_{1}, \ldots,  j_{k} \geq 0, \sum j_{i} = m - k} \sum_{i_{1} = 1}^{n} \sum_{i_{2} = 1}^{n} \cdots \sum_{i_{k} = 1}^{n} \lambda_{i_{1}}^{j_{1}} \lambda_{i_{2}}^{j_{2}} \cdots \lambda_{i_{k}}^{j_{k}} B_{i_{1} i_{2}} B_{i_{2} i_{3}} \cdots B_{i_{k} i_{1}}.
		\end{align*}
		But since
		\begin{align*}
			\sum_{j_{1}, \ldots,  j_{k} \geq 0, \sum j_{i} = m - k} \lambda_{i_{1}}^{j_{1}} \lambda_{i_{2}}^{j_{2}} \cdots \lambda_{i_{k}}^{j_{k}} = [\lambda_{i_{1}}, \lambda_{i_{2}}, \ldots, \lambda_{i_{k}}]_{x^{m - 1}},
		\end{align*}
		we are done.
	\end{proof}

	Let us now focus on the main expression (\ref{main_expr}) with $k = 4$ and $f = |\cdot|^{3}$. Write $g = f'/3 = (\cdot) |\cdot|$ We have the following simple identities.
	\begin{lem}
		Let $a_{1}, a_{2}, a_{3}, a_{4} < 0 < b_{1}, b_{2}, b_{3}, b_{4}$. Then the following identities hold:
		\begin{align}
			[a_{1}, a_{2}, a_{3}, a_{4}]_{g} =& 0 \label{easy_terms1}\\
			[a_{1}, a_{2}, a_{3}, b_{1}]_{g} =& \frac{2 b_{1}^2}{(b_{1} - a_{1}) (b_{1} - a_{2}) (b_{1} - a_{3})} \label{mid1}\\
			[a_{1}, a_{2}, b_{1}, b_{2}]_{g} =& \frac{2 (b_{1} + b_{2}) a_{1} a_{2} - 2 (a_{1} + a_{2}) b_{1} b_{2}}{(b_{1} - a_{1}) (b_{1} - a_{2}) (b_{2} - a_{1}) (b_{2} - a_{2})} \label{midmid}\\
			=& \frac{-2 a_{1} b_{1}}{(b_{1} - a_{1}) (b_{1} - a_{2}) (b_{2} - a_{1})} + \frac{-2 a_{2} b_{2}}{(b_{2} - a_{2}) (b_{2} - a_{1}) (b_{1} - a_{2})} \nonumber \\
			[a_{1}, b_{1}, b_{2}, b_{3}]_{g} =& \frac{2 a_{1}^2}{(b_{1} - a_{1}) (b_{2} - a_{1}) (b_{3} - a_{1})} \label{mid2} \\
			[b_{1}, b_{2}, b_{3}, b_{4}]_{g} =& 0. \label{easy_terms2}
		\end{align}
		\begin{proof}
			Straightforward to check from the definitions.
		\end{proof}
	\end{lem}

	Let us then complete the proof of Theorem \ref{posres}. Assume that the matrix $A$ has $n_{1}$ negative ($\lambda_{1}, \lambda_{2}, \ldots, \lambda_{n_{1}}$) and $n_{2}$ positive ($\lambda_{n_{1} + 1}, \ldots, \lambda_{n}$) eigenvalues ($n_{1} + n_{2} = n$). To prove (\ref{main_expr}) is non-negative, we first group it in several parts. Sum consists of $n^{4}$ summands; separate them in patterns depending on whether $i_{1}$, $i_{2}$, $i_{3}$ and $i_{4}$ are at most or greater than $n_{1}$. If say $i_{1}, i_{3} \leq n_{1} < i_{2}, i_{4}$, assign this term pattern $(-, +, -, +)$. Write then $S_{-, +, -, +}$ for the sum of all the terms with pattern $(-, +, -, +)$. Note that since $[x_{0}, x_{1}, \ldots, x_{k}]_{f} = [x_{\sigma(0)}, x_{\sigma(1)}, \ldots, x_{\sigma(k)}]_{f}$ for any permutation $\sigma$, many of these sums are equal, and they can be in fact split into following groups:
	\begin{enumerate}[(a)]
		\item $S_{-, -, -, -}$
		\item $S_{+, -, -, -}$, $S_{-, +, -, -}$, $S_{-, -, +, -}$, $S_{-, -, -, +}$
		\item $S_{+, +, -, -}$, $S_{-, +, +, -}$, $S_{-, -, +, +}$, $S_{+, -, -, +}$
		\item $S_{+, -, +, -}$, $S_{-, +, -, +}$
		\item $S_{-, +, +, +}$, $S_{+, -, +, +}$, $S_{+, +, -, +}$, $S_{+, +, +, -}$
		\item $S_{+, +, +, +}$.
	\end{enumerate}
	Note also that by (\ref{easy_terms1}) and (\ref{easy_terms2}), $S_{-, -, -, -} = 0 = S_{+, +, +, +}$. Moreover, by (\ref{mid1}), (\ref{midmid}) and (\ref{mid2}), we have
	\begin{align*}
		S_{+, -, -, -} =& \sum_{i_{1} = n_{1} + 1}^{n} \sum_{i_{2}, i_{3}, i_{4} = 1}^{n_{1}} \frac{2 \lambda_{i_{1}}^2 B_{i_{1} i_{2}} B_{i_{2} i_{3}} B_{i_{3} i_{4}} B_{i_{4} i_{1}}}{( \lambda_{i_{1}} -  \lambda_{i_{2}}) ( \lambda_{i_{1}} -  \lambda_{i_{3}}) (\lambda_{i_{1}} -  \lambda_{i_{4}})} \\
		S_{+, +, -, -} =&\sum_{i_{1}, i_{2} = n_{1} + 1}^{n} \sum_{i_{3}, i_{4} = 1}^{n_{1}} \frac{-2 \lambda_{i_{1}} \lambda_{i_{3}} B_{i_{1} i_{2}} B_{i_{2} i_{3}} B_{i_{3} i_{4}} B_{i_{4} i_{1}}}{(\lambda_{i_{1}} - \lambda_{i_{3}}) (\lambda_{i_{2}} -\lambda_{i_{3}}) (\lambda_{i_{1}} - \lambda_{i_{4}})} \\
		+& \sum_{i_{1}, i_{2} = n_{1} + 1}^{n} \sum_{i_{3}, i_{4} = 1}^{n_{1}} \frac{-2 \lambda_{i_{2}} \lambda_{i_{4}} B_{i_{1} i_{2}} B_{i_{2} i_{3}} B_{i_{3} i_{4}} B_{i_{4} i_{1}}}{(\lambda_{i_{2}} - \lambda_{i_{4}}) (\lambda_{i_{1}} -\lambda_{i_{4}}) (\lambda_{i_{2}} - \lambda_{i_{3}})}
	\end{align*}
	\begin{align*}
		S_{+, -, +, -} =& \sum_{i_{1}, i_{3} = n_{1} + 1}^{n} \sum_{i_{2}, i_{4} = 1}^{n_{1}} \frac{-2 \lambda_{i_{1}} \lambda_{i_{3}} (\lambda_{i_{2}} + \lambda_{4}) B_{i_{1} i_{2}} B_{i_{2} i_{3}} B_{i_{3} i_{4}} B_{i_{4} i_{1}}}{(\lambda_{i_{1}} - \lambda_{i_{2}}) (\lambda_{i_{3}} -\lambda_{i_{2}}) (\lambda_{i_{1}} - \lambda_{i_{4}}) (\lambda_{i_{3}} - \lambda_{i_{4}})} \\
		+& \sum_{i_{1}, i_{3} = n_{1} + 1}^{n} \sum_{i_{2}, i_{4} = 1}^{n_{1}} \frac{2 \lambda_{i_{2}} \lambda_{i_{4}} (\lambda_{i_{1}} + \lambda_{3}) B_{i_{1} i_{2}} B_{i_{2} i_{3}} B_{i_{3} i_{4}} B_{i_{4} i_{1}}}{(\lambda_{i_{1}} - \lambda_{i_{2}}) (\lambda_{i_{3}} -\lambda_{i_{2}}) (\lambda_{i_{1}} - \lambda_{i_{4}}) (\lambda_{i_{3}} - \lambda_{i_{4}})} \\
		S_{-, +, +, +} =& \sum_{i_{1} = 1}^{n_{1}} \sum_{i_{2}, i_{3}, i_{4} = n_{1} + 1}^{n} \frac{2 \lambda_{i_{1}}^2 B_{i_{1} i_{2}} B_{i_{2} i_{3}} B_{i_{3} i_{4}} B_{i_{4} i_{1}}}{( \lambda_{i_{2}} -  \lambda_{i_{1}}) ( \lambda_{i_{3}} -  \lambda_{i_{1}}) (\lambda_{i_{4}} -  \lambda_{i_{1}})}.
	\end{align*}
	Our goal is to prove that
	\begin{align*}
		(\ref{main_expr}) = 4 S_{+, -, -, -} + 4 S_{+, +, -, -} + 2 S_{+, -, +, -} + 4 S_{-, +, +, +} \geq 0.
	\end{align*}
	We shall in fact prove that
	\begin{align}
		S_{+, -, -, -} + S_{+, +, -, -} + S_{-, +, +, +} \geq &0 \label{harder_thing} \\
		\textrm{and}& \nonumber \\
		S_{+, -, +, -} \geq &0 \label{easier_thing}.
	\end{align}
	The inequalities (\ref{harder_thing}) and (\ref{easier_thing}) correspond to terms where the (cyclic) sign sequence changes sign $2$ and $4$ times, respectively. To prove (\ref{harder_thing}), first note that (simply relabel the indices, factor, and use the Hermitian property $B_{j, i} = \overline{B_{i, j}}$)
	\begin{align*}
		S_{+, -, -, -} =& \sum_{i > n_{1} \geq j} \frac{2 \lambda_{i}^{2}}{\lambda_{i} - \lambda_{j}}\left|\sum_{l = 1}^{n_{1}} \frac{B_{i, l} B_{l, j}}{\lambda_{i} - \lambda_{l}}\right|^{2} \\
		S_{-, +, +, +} =& \sum_{i > n_{1} \geq j} \frac{2 \lambda_{j}^{2}}{\lambda_{i} - \lambda_{j}}\left|\sum_{l = n_{1} + 1}^{n} \frac{B_{i, l} B_{l, j}}{\lambda_{l} - \lambda_{j}}\right|^{2}.
	\end{align*}
	In a similar vein,
	\begin{align*}
		S_{+, +, -, -} =& \sum_{i > n_{1} \geq j} \frac{-2 \lambda_{i} \lambda_{j}}{\lambda_{i} - \lambda_{j}} \left(\sum_{l = n_{1} + 1}^{n} \frac{B_{i, l} B_{l, j}}{\lambda_{l} - \lambda_{j}} \right) \overline{\left(\sum_{l = 1}^{n_{1}} \frac{B_{i, l} B_{l, j}}{\lambda_{i} - \lambda_{l}} \right)} \\
		+& \sum_{i > n_{1} \geq j} \frac{-2 \lambda_{i} \lambda_{j}}{\lambda_{i} - \lambda_{j}} \left(\sum_{l = 1}^{n_{1}} \frac{B_{i, l} B_{l, j}}{\lambda_{i} - \lambda_{l}} \right) \overline{\left(\sum_{l = n_{1} + 1}^{n} \frac{B_{i, l} B_{l, j}}{\lambda_{l} - \lambda_{j}} \right)}.
	\end{align*}
	But this just means that
	\begin{align}\label{form1}
		S_{+, -, -, -} + S_{+, +, -, -} + S_{-, +, +, +} = 2 \sum_{i > n_{1} \geq j} \frac{1}{\lambda_{i} - \lambda_{j}} \left|\sum_{l = 1}^{n_{1}} \frac{\lambda_{i} B_{i, l} B_{l, j}}{\lambda_{i} - \lambda_{l}} + \sum_{l = n_{1} + 1}^{n} \frac{ \lambda_{j} B_{i, l} B_{l, j}}{\lambda_{j} - \lambda_{l}} \right|^{2} \geq 0.
	\end{align}
	Quite similarly
	\begin{align}\label{form2}
		S_{+, -, +, -} =& \sum_{i_{1}, i_{2} = 1}^{n_{1}} (-\lambda_{i_{1}} - \lambda_{i_{2}}) \left|\sum_{l = n_{1} + 1}^{n} \frac{\lambda_{l} B_{i_{1}, l} B_{l, i_{2}}}{(\lambda_{i_{1}} - \lambda_{l})(\lambda_{i_{2}} - \lambda_{l})} \right|^{2} \\ \nonumber
		+& \sum_{i_{1}, i_{2} = n_{1} + 1}^{n} (\lambda_{i_{1}} + \lambda_{i_{2}}) \left|\sum_{l = 1}^{n_{1}} \frac{\lambda_{l} B_{i_{1}, l} B_{l, i_{2}}}{(\lambda_{i_{1}} - \lambda_{l})(\lambda_{i_{2}} - \lambda_{l})} \right|^{2} \geq 0.
	\end{align}
	The proof is thus complete.
\end{proof}

\begin{proof}[Proof of Theorem \ref{mainemb}]
    By Theorem \ref{posres} and Lemmas \ref{lemKol} and \ref{lem_decay} for any Hermitian $n \times n$ matrices $A, B$ with $B$ invertible there exists a finite measure on $S^{1}$ such that
    \begin{align*}
        \int_{S^{1}} |\langle v, w \rangle|^{3} d \mu (w) = \tr |v_{1} A + v_{2} B|^{3} = \|v_{1} A + v_{2} B\|_{3}^{3}.
    \end{align*}
    By simple approximation we see that we can drop the invertibility assumption; indeed on LHS one simply takes a limit point of measures of the approximant, which exists by compactness. For non-symmetric matrices, use the standard isometry trick
    \begin{align*}
        A \mapsto \frac{1}{2^{1/3}}\begin{bmatrix}
        0 & A \\
        A^{*} & 0
        \end{bmatrix}
    \end{align*}
    to note that any plane is isometric to a Hermitian plane. Finally for $A, B \in \S_{3}$, approximate with finite rank operators and apply the matrix case.
\end{proof}

\begin{kor}
    Hanner's inequality (\ref{hanner1}) holds when $p = 3$ and (in reverse when) $p = 3/2$.
\end{kor}

\begin{proof}
    Take linear isometric embedding $T : \vspan_{\R}\{A, B\} \mapsto L_{3}$. Now by the isometry property and linearity
    \begin{align*}
        & \|A + B\|_{\S_3}^{3} + \|A - B\|_{\S_3}^{3} \leq (\|A\|_{\S_3} + \|B\|_{\S_3})^{3} + |\|A\|_{\S_3} - \|B\|_{\S_3}|^{3} \\
        \Leftrightarrow & \|T A + T B\|_{3}^{3} + \|T A - T B\|_{3}^{3} \leq (\|T A\|_{3} + \|T B\|_{3})^{3} + |\|TA\|_{3} - \|T B\|_{3}|^{3}.
    \end{align*}
    But the second inequality follows from Hanner's original result for $L_{p}$ \cite{hanner1956uniform}. The case $p = 3/2$ follows from duality (see \cite[Lemma $6$]{ball1994sharp}).
\end{proof}

\begin{proof}[Proof of Theorem \ref{negthm}]
    Assume first that $p$ is not an even integer. Towards a contradiction, assume that there exists a measure $\mu$
    \begin{align*}
        \left\|
        \begin{bmatrix}
            z + x & y \\
            y & z - x
        \end{bmatrix}
        \right\|_{\S_{p}} = \left|z + \sqrt{x^{2} + y^{2}}\right|^{p} + \left|z - \sqrt{x^{2} + y^{2}}\right|^{p} = \int_{S^{2}} |x t_{1} + y t_{2} + z t_{3}|^{p} d \mu (t_{1}, t_{2}, t_{3}).
    \end{align*}
    Setting $(x, y) = (r \cos(\theta), r \sin(\theta))$ we see that
    \begin{align}\label{3spaceformula}
        \left|z + r\right|^{p} + \left|z - r\right|^{p} = \int_{S^{2}} |r(\cos(\theta) t_{1} + \sin(\theta) t_{2}) + z t_{3}|^{p} d \mu (t_{1}, t_{2}, t_{3}).
    \end{align}
    Note however that by the uniqueness (for $2$-dimensional $L_{p}$-spaces), the pushforward $f_{*}(\mu)$ with
    \begin{align*}
        f :(t_{1}, t_{2}, t_{3}) \mapsto \frac{(\cos(\theta) t_{1} + \sin(\theta) t_{2}, t_{3})}{\|(\cos(\theta) t_{1} + \sin(\theta) t_{2}, t_{3})\|_{2}}
    \end{align*}
    should map the support of $\mu$ to $(\pm 1/\sqrt{2}, \pm 1/\sqrt{2})$, which clearly cannot be true for every $\theta$.
    
    Consider then an even integer $p \geq 4$. Again, it is enough to refute the existence of measure $\mu$ on $S^{2}$ such that (\ref{3spaceformula}) holds. Expanding both sides around $r = 0$ and averaging over $\theta \in [0, 2 \pi]$ we see that
    \begin{align*}
        \mu(t_{3}^{p}) &= 2 \\
        \mu(t_{3}^{p - 2} (1 - t_{3}^{2}))) &= 4 \\
        \mu(t_{3}^{p - 4} (1 - t_{3}^{2})^{2}) &= \frac{16}{3}.
    \end{align*}
    But now $\mu((2 t_{3}^{2} - 1)^{2} t_{3}^{p - 4}) = -2/3$, which is impossible.
\end{proof}

\subsection{Proofs of Theorems \ref{der4} and \ref{der3}}

\begin{proof}[Proof of Theorem \ref{der4}]
    By \cite[Corollary 8]{bullen1971criterion}, functions with non-negative 4th derivative are (on every compact interval) pointwise limits of functions of the form
    \begin{align}\label{der4discrete}
        t \mapsto p(t) + \sum_{i} w_{i} |t - c_{i}|^{3},
    \end{align}
    where $p$ is a polynomial of degree at most $3$ and $w_{i} \geq 0$. But Theorem \ref{posres} readily implies that for any such function the resulting trace function has non-negative 4th derivative, at least if $B$ is invertible. Taking limit along invertible approximants of $B$ and approximants of the form (\ref{der4discrete}) yields the claim.
\end{proof}

To prove Theorem \ref{der3} we modify proof of Theorem \ref{posres}.

\begin{lause}\label{posres3}
	For any Hermitian $n \times n$ matrices $A$ and $B$, with $B$ positive definite, the function
	\begin{align*}
		T := t \mapsto \tr |A + t B| (A + t B)
	\end{align*}
	is analytic outside finitely many points where $A + t B$ is singular. Outside these points the function has non-negative $3$rd derivative. At the singular points the function behaves like
	\begin{align*}
		C |t - c| (t - c) + D |t - c| (t - c)^{2} + f(t),
	\end{align*}
	where $C \geq 0$ and $f$ is $C^{3}$ near $c$. Consequently the $3$rd derivative has non-negative multiple of $\delta$ measure at the singular points is therefore altogether non-negative measure.
\end{lause}

\begin{proof}
    The non-smooth part can be handled as in the proof of Theorem \ref{posres}. Thanks to Lemma \ref{derivative_identity}, we only need to understand
    \begin{align*}
        \sum_{i_{1} = 1}^{n} \sum_{i_{2} = 1}^{n} \sum_{i_{3} = 1}^{n} [\lambda_{i_{1}}, \lambda_{i_{2}}, \lambda_{i_{3}}]_{|\cdot|} B_{i_{1}, i_{2}} B_{i_{2}, i_{3}} B_{i_{3}, i_{1}}.
    \end{align*}
    \begin{lem}
        Let $a_{1}, a_{2}, a_{3} < 0 < b_{1}, b_{2}, b_{3}$. Then we have the following identities:
        \begin{align*}
            [a_{1}, a_{2}, a_{3}]_{|\cdot|} &= 0 \\
            [a_{1}, a_{2}, b_{1}]_{|\cdot|} &= \frac{2 b_{1}}{(b_{1} - a_{1}) (b_{1} - a_{2})} \\
            [a_{1}, b_{1}, b_{2}]_{|\cdot|} &= \frac{-2 a_{1}}{(b_{1} - a_{1}) (b_{2} - a_{1})} \\
            [b_{1}, b_{2}, b_{3}]_{|\cdot|} &= 0
        \end{align*}
    \end{lem}
    \begin{proof}
        Straightforward calculation.
    \end{proof}
    Assume again that the first $n_{1}$ eigenvalues of $A$ are negative and the rest $n_{2}$ are positive. By cyclicity it suffices to understand
    \begin{align*}
        S_{-, +, +} &= \sum_{i_{1} = 1}^{n_{1}} \sum_{i_{2} = n_{1} + 1}^{n} \sum_{i_{3} = n_{1} + 1}^{n} \frac{-2 \lambda_{i_{1}}}{(\lambda_{i_{2}} - \lambda_{i_{1}}) (\lambda_{i_{3}} - \lambda_{i_{1}})} B_{i_{1}, i_{2}} B_{i_{2}, i_{3}} B_{i_{3}, i_{1}} \text{ and} \\
        S_{+, -, -} &= \sum_{i_{1} = n_{1} + 1}^{n} \sum_{i_{2} = 1}^{n_{1}} \sum_{i_{3} = 1}^{n_{1}} \frac{2 \lambda_{i_{1}}}{(\lambda_{i_{1}} - \lambda_{i_{2}}) (\lambda_{i_{1}} - \lambda_{i_{3}})} B_{i_{1}, i_{2}} B_{i_{2}, i_{3}} B_{i_{3}, i_{1}}.
    \end{align*}
    But these are both non-negative as we can write
    \begin{align*}
        S_{-, +, +} = 2\sum_{i \leq n_{1}} -\lambda_{i} \langle B v_{i}, v_{i} \rangle \\
        S_{+, -, -} = 2\sum_{i > n_{1}} \lambda_{i} \langle B u_{i}, u_{i} \rangle,
    \end{align*}
    where
    \begin{align*}
        v_{i} = \sum_{j > n_{1}} \frac{B_{j, i}}{\lambda_{j} - \lambda_{i}} e_{j} \\
        u_{i} = \sum_{j \leq n_{1}} \frac{B_{j, i}}{\lambda_{i} - \lambda_{j}} e_{j},
    \end{align*}
    and $e_{i}$ is the eigenvector of $A$ corresponding to eigenvalue $\lambda_{i}$.
\end{proof}

\begin{proof}[Proof of Theorem \ref{der3}]
    Make obvious modifications to the proof of Theorem \ref{der4}.
\end{proof}

\section{Discussion}\label{discuss}

It would be extremely interesting to know whether conjecture \ref{mainconj} is true in general. In addition to the case $p = 3$, it is also clearly true for $p = 2$ ($\S_{2}$ is Hilbert space) and $p = \infty$ ($L_{\infty}$ is universal). The case $p = 1$ is also well known; every $2$-dimensional normed space is isometric to a subspace of $L_{1}$, as observed by Lindenstrauss \cite[Corollary $2$]{lindenstrauss1964extension}.

In the following discussion we shall assume that $A$ and $B$ are Hermitian $n \times n$ matrices.

\subsection{$p$ is not an integer}

To prove conjecture \ref{mainconj}, one might try to modify the proof of Theorem \ref{mainemb}. Lemma \ref{lemKol} has generalization as follows. Define
\begin{align*}
    c_{p} = 2^{p + 1} \pi^{1/2} \Gamma((p + 1)/2)/\Gamma(-p/2).
\end{align*}
\begin{lem}\label{generalLemKol}
    Let $p \geq 1$ be not an even integer and let $\|\cdot\|$ be a norm in $\R^{2}$ such that for any $v, w \in \R^{2}$ the $(p + 1)$:st distributional derivative of
	\begin{align*}
		t \mapsto (1/c_{p}) \|t v + w\|^{p}
	\end{align*}
	is a finite positive Borel measure $\mu$, for which $\mu(|x|^p) < \infty$. Then $\|\cdot\|$ is linearly isometrically embeddable into $L_{p}$.
\end{lem}

Here by distributional derivative of $f$ order $\alpha$ we mean the distribution $\widehat{(|\cdot|^{\alpha} \hat{f})}$ ($\hat{\cdot}$ stands for the Fourier transform). Exact value of the constant $c_{p}$ here is of course irrelevant. Only the sign matters, which (that of $\Gamma(-p/2)$) happens to be $+/-$ depending on if $\lfloor p/2 \rfloor$ is odd/even. The proof of Lemma \ref{generalLemKol} is omitted, since we don't know how to generalize the remaining steps.

\subsection{$p$ is an odd integer}

Lemma \ref{generalLemKol} is particularly convenient if $p$ is odd integer as then $(p + 1)$:st fractional derivative corresponds to the usual (distributional) derivative, up to a constant. For odd integers $p > 3$ one has expressions similar to (\ref{main_expr}); Lemma \ref{lem_decay} goes also through without noticeable change. The main difficulty then is to find expressions similar to (\ref{form1}) and (\ref{form2}). We have partial work, which suggests that the Conjecture \ref{mainconj} is true for $p = 5$.

If one restricts attention to $2 \times 2$ matrices, the case of an odd integer can be checked without too much trouble. Indeed, one checks that
\begin{align*}
    (\tr |A + t B|^{p})^{(p + 1)}(0) = \sum_{\substack{i, j \geq 0 \\ 2 i + j = p - 1}} \frac{(p - 1)!}{j! i! (i + 1)!} \left(\frac{B_{1,1} \lambda_{2} - B_{2,2} \lambda_{1}}{\lambda_{2} - \lambda_{1}}\right)^{j} B_{1, 2}^{2 i + 2} \left(\frac{- \lambda_{1} \lambda_{2}}{(\lambda_{2} - \lambda_{1})^{2}}\right)^{i}
\end{align*}
whenever $\lambda_{1} < 0 < \lambda_{2}$. Here all the $(p + 1)/2$ terms are clearly non-negative, and, akin to the (general $n$) case of $p = 3$, they correspond to different numbers of sign changes for the patterns introduced in the proof Theorem \ref{posres}. Sadly for $p > 3$ and general $n$ such decomposition to $(p + 1)/2$ non-negative parts is not possible.

\subsection{$p$ is an even integer}

These cases are somewhat special. While now there's no more uniqueness in the representation (\ref{lp_repr}), one can simply expand $\|A + t B \|_{\S_{p}}^{p}$ as a polynomial in $t$. Equating the coefficients (in $t$) of
\begin{align*}
    \int_{0}^{2 \pi} |\cos(\theta) + \sin(\theta) t|^{p} d \mu(\theta) = \tr |A + t B|^{p},
\end{align*}
one ends up with expressions of the Fourier coefficients of $\mu$, $\hat{\mu}(0), \hat{\mu}(2), \hat{\mu}(4), \ldots, \hat{\mu}(p)$ in terms of traces of non-commutative polynomials of $A$ and $B$. Since we assume that $\mu$ is symmetric, we may instead think about $\mu$ on $[0, \pi]$ in which case one obtains expressions for, $\hat{\mu}(0), \hat{\mu}(1), \hat{\mu}(2), \ldots, \hat{\mu}(p/2)$

For $p = 2$, this yields
\begin{align*}
    \hat{\mu}(0) &= \tr(A^2) + \tr(B^2) \\
    \hat{\mu}(1) &= \tr(A^2) - \tr(B^2) + 2i \tr(AB).
\end{align*}
The necessary and sufficient condition for the existence of $\mu$ ($\hat{\mu}(0) \geq |\hat{\mu}(1)|$) simplifies to
\begin{align*}
    \tr(A B)^{2} \leq \tr(A)^2 \tr(B)^2,
\end{align*}
Cauchy--Schwarz inequality for the inner product $\langle A, B \rangle = \tr(A B)$.

For $p = 4$ the situation is more complicated. One has
\begin{align*}
    \hat{\mu}(0) &= \tr(A^4) + \frac{2}{3}\left(2 \tr(A^2 B^{2}) + \tr(A B A B) \right) + \tr(B^{4}) \\
    \hat{\mu}(1) &= \tr(A^4) - \tr(B^{4}) + 2 i \left(\tr(A^3 B) + \tr(A B^{3}) \right) \\
    \hat{\mu}(2) &= \tr(A^4) + \tr(B^{4}) - 2 \left(2 \tr(A^2 B^{2}) + \tr(A B A B) \right) + 4 i \left(\tr(A^3 B) - \tr(A B^{3}) \right).
\end{align*}
Existence of the measure $\mu$ is now equivalent to (see for instance \cite[Theorem 1.3.6]{bakonyi2011matrix})
\begin{align*}
    \begin{bmatrix}
        \hat{\mu}(0) & \hat{\mu}(1) & \hat{\mu}(2) \\
        \overline{\hat{\mu}(1)} & \hat{\mu}(0) & \hat{\mu}(1) \\
        \overline{\hat{\mu}(2)} & \overline{\hat{\mu}(1)} & \hat{\mu}(0)
    \end{bmatrix} \geq 0,
\end{align*}
i.e. $\mu(p) \geq 0$ for any trigonometric polynomial $p$ of degree 2 non-negative on $S^{1}$. By \cite[Theorem 1.1.7]{bakonyi2011matrix} it is enough to check such polynomials with roots on the unit circle, namely that
\begin{align*}
    0 \leq & \mu(\theta \mapsto (e^{i \theta} - e^{i \theta_{1}}) (e^{i \theta} - e^{i \theta_{2}}) (e^{-i \theta} - e^{-i \theta_{1}}) (e^{-i \theta} - e^{-i \theta_{1}})) \\
    =& (4 + e^{i (\theta_{1} - \theta_{2})} + e^{i (\theta_{2} - \theta_{1})}) \hat{\mu}(0) \\
    &- 2 (e^{-i \theta_{1}} + e^{-i \theta_{2}}) \hat{\mu}(1) - 2 (e^{i \theta_{1}} + e^{i \theta_{2}}) \overline{\hat{\mu}(1)} \\
    &+ e^{-i (\theta_{1} + \theta_{2})} \hat{\mu}(2) + e^{i (\theta_{1} + \theta_{2})} \overline{\hat{\mu}(2)}
\end{align*}
for any $\theta_{1}, \theta_{2} \in \R$.

But since this can be expressed as
\begin{align*}
    & (2 - e^{i \theta_{1}} - e^{-i \theta_{1}}) (2 - e^{i \theta_{2}} - e^{-i \theta_{2}}) \tr(A^{4}) \\
    + & 4 i \left( e^{i \theta_{1}} + e^{i \theta_{2}} - e^{i (\theta_{1} + \theta_{2})} - e^{-i \theta_{1}} - e^{-i \theta_{2}} + e^{-i (\theta_{1} + \theta_{2})} \right) \tr(A^{3} B) \\
    +& \left(8 + 2 (e^{i (\theta_{1} - \theta_{2})} + e^{i (\theta_{2} - \theta_{1})}) - 6 (e^{i (\theta_{1} + \theta_{2})} + e^{-i (\theta_{2} + \theta_{1})}) \right)\frac{2 \tr(A^2 B^2) + \tr(A B A B)}{3}\\
   + & 4 i \left( e^{i \theta_{1}} + e^{i \theta_{2}} + e^{i (\theta_{1} + \theta_{2})} - e^{-i \theta_{1}} - e^{-i \theta_{2}} - e^{-i (\theta_{1} + \theta_{2})} \right) \tr(A B^{3}) \\
   + & (2 + e^{i \theta_{1}} + e^{-i \theta_{1}}) (2 + e^{i \theta_{2}} + e^{-i \theta_{2}}) \tr(A^{4}) \\
   =& \|(e^{i \theta_{1}} - 1) (e^{i \theta_{2}} - 1) A^{2} - i (e^{i (\theta_{1} + \theta_{2})} - 1) (AB + BA) -(e^{i \theta_{1}} + 1) (e^{i \theta_{2}} + 1) B^{2}  \|_{\S_{2}}^{2} \\
   +& \frac{|e^{i \theta_{1}} - e^{i \theta_{2}}|^{2}}{3} \|A B - B A\|_{\S_{2}}^{2} \\
   \geq & 0,
\end{align*}
we obtain 
\begin{lause}
     For any $A, B \in \S_{4}$ the real space $\vspan_{\R}\{A, B\}$ can be linearly isometrically embedded into $L_{4}$.
\end{lause}

Similar but more complicated argument can be carried out for $p = 6$, but it seems that such sum-of-squares expression is impossible for $p = 8$.

\subsection{Stahl's Theorem}

We give a quick proof of Stahl's result \ref{stahl} given Conjecture \ref{der_conjecture}.

\begin{proof}[Proof of Theorem \ref{stahl}]
    By the result of Bernstein on completely monotone functions (see \cite[Theorem 1.4]{schilling2012bernstein}) Theorem is equivalent to the function
    \begin{align*}
        t \mapsto \tr(\exp(A + t B))
    \end{align*}
    having non-negative $k$'th derivative for any $k \geq 0$. But since exponential function has non-negative $k$'th derivative for any $k$, Conjecture \ref{der_conjecture} immediately implies the claim.
\end{proof}

\subsection{Further questions}

We finish with two questions.

\begin{kys}
    Let $X$ be a subspace of $\S_{p}$. We say that $X$ is of $p$-width $k$, if it is finitely representable\footnote{$X$ is said to finitely representable in $Z$ if for any finite dimensional subspace $Y$ of $X$ and $\varepsilon > 0$ there exists a subspace $Z_{Y} \subset Z$ and isomorphism $T : Y \to Z_{Y}$ such that $\|T\| \|T^{-1}\| \leq 1 + \varepsilon$.} in $(\oplus_{i = 1}^{\infty} \S_{p}^{k})_{p}$. Does there exist $1 \leq p < \infty$, $p \neq 2$, and $m, k \in \N$ with $m \geq 3$ such that any $m$-dimensional subspace of $\S_{p}$ has $p$-width at most $k$?
\end{kys}

\begin{kys}
    Let $A, B \in \S_{p}$ (not necessarily Hermitian). Does there exist $f, g \in L_{1}((0, 1), \C)$ such that
    \begin{align*}
        \|\alpha A + \beta B\|_{\S_{p}} = \|\alpha f + \beta g\|_{p}
    \end{align*}
    for any $\alpha, \beta \in \C$? In other words, is $\vspan_{\C} \{A, B\}$ complex linearly isometric to a subspace of complex $L_{p}$?
\end{kys}

\subsection*{Acknowledgements}

The author wishes to thank Assaf Naor for introducing the topic, and Assaf Naor and Victoria Chayes for helpful discussions.

\bibliography{schatten_hanner_3}

\bibliographystyle{abbrv}

\end{document}